\documentclass[12pt]{amsart}

\usepackage{amsmath}
\usepackage{amssymb}
\usepackage{amsthm}
\usepackage{amscd}
\usepackage{xypic}
\headsep=1cm \numberwithin{equation}{section}
\def\P{{\mathbb P}}

\newtheorem{theorem}{Theorem}[section]

\newtheorem{proposition}[theorem]{Proposition}

\theoremstyle{definition}

\newtheorem{definition and remark}[theorem]{Definition and Remark}
\newtheorem{remark}[theorem]{Remark}
\newtheorem{remark and definition}[theorem]{Remark and Definition}
\newtheorem{remark and notation}[theorem]{Remark and Notation}
\newtheorem{notation and remark}[theorem]{Notation and Remark}
\newtheorem{notation and convention}[theorem]{Notation and Convention}

\newtheorem{notation and remarks}[theorem]{Notation and Remarks}
\newtheorem{notation and reminder}[theorem]{Notation and Reminder}
\newtheorem{example}[theorem]{Example}

\newtheorem{construction and examples}[theorem]{Construction and Examples}

\newtheorem{problem and remark}[theorem]{Problem and Remark}

\title{On the rank index of some quadratic varieties}

\begin{document}

\author{Hyunsuk Moon}
\address{Hyunsuk Moon,KIAS, Seoul 02455, Republic of Korea}
\email{hsmoon@kias.re.kr}

\author{Euisung Park }
\address{Department of Mathematics, Korea University, Seoul 02841, Republic of Korea}
\email{euisungpark@korea.ac.kr}

\date{Seoul, \today}



\thispagestyle{empty} \maketitle

\setcounter{page}{1}

\begin{abstract}
Regarding the generating structure of the homogeneous ideal of a projective variety $X \subset \P^r$, we define the rank index of $X$ to be the smallest integer $k$ such that $I(X)$ can be generated by quadratic polynomials of rank at most $k$. Recently it is shown in \cite{HLMP} that every Veronese embedding has rank index $3$ if the base field has characteristic $\ne 2, 3$. In this paper, we introduce some basic ways of how to calculate the rank index and find its values when $X$ is some other classical projective varieties such as rational normal scrolls, del Pezzo varieties, Segre varieties and the Pl\"{u}cker embedding of the Grassmannian of lines.
\end{abstract}

\setcounter{page}{1}

\section{Introduction}
\noindent Let $X \subset \P^r$ be a nondegenerate irreducible projective variety over an algebraically closed field $\mathbb{K}$ and let $I(X)$ be the homogeneous ideal of $X$ in the homogeneous coordinate ring $R = \mathbb{K}[z_0 , z_1 , \ldots,z_r]$ of $\P^r$. Describing generators of $I(X)$ is a basic problem to understand algebraic properties of $X$. We say that $X$ is a quadratic variety if $I(X)$ is generated by quadratic polynomials. Over the past few decades, several structural features of $I(X)$ such as Green-Lazarsfeld's property $N_p$ (cf. \cite{Gre1}, \cite{Gre2}, \cite{GL}, \cite{EL}, \cite{GP}, \cite{Ina}, etc) and the determinantal presentation (cf. \cite{EKS}, \cite{Pu}, \cite{Ha}, \cite{B}, \cite{SS}, etc) have attracted considerable attention. Recently, the authors in \cite{HLMP} begin to study the finer structure of the finite dimensional $\mathbb{K}$-vector space $I(X)_2$ of quadratic defining equations of $X$ through the rank of its members. More precisely, we say that $X$ satisfies property $QR(k)$ if $I(X)$ can be generated by quadrics of rank at most $k$. Also the rank index of $X$, denoted by $\mbox{rank-index}(X)$, is defined to be the smallest integer $k$ such that $X$ satisfies property $QR(k)$. Obviously, the rank index of $X$ is preserved under the coordinate change of $\P^r$. Also, it is at least $3$ since $X$ is irreducible and nondegenerate. There are many known cases such that the rank index is at most $4$. For example, several classical constructions in projective geometry, such as rational normal scrolls, Veronese varieties and Segre varieties of more than two projective spaces, have rank index at most $4$.

Turning to the property $QR(3)$, it is shown in \cite{HLMP} that the rank index of the $d$th Veronese embedding of $\P^n$ is equal to $3$ when ${\rm char} (\mathbb{K}) \neq 2,3$ and when ${\rm char} (\mathbb{K}) = 3$ and $n=1$ or $n=d=2$. The rank index of the second Veronese embedding of $\P^n$ is equal to $4$ for all $n \geq 3$ if ${\rm char} (\mathbb{K}) = 3$. Thus it remains only if ${\rm char} (\mathbb{K}) = 3$, $n \geq 2$ and $d \geq 3$.

The main purpose of the present paper is to find the rank index of some other classical projective varieties such as rational normal scrolls, del Pezzo varieties, Segre varieties and the Pl\"{u}cker embedding of the Grassmannian of lines. For these varieties, the upper bound of the rank index is already well known. For example, the homogeneous ideal of the Pl\"{u}cker embedding of the Grassmannian of lines is generated by the Pl\"{u}cker relations which are quadratic polynomials of rank $6$ (cf. \cite[Theorem 8.15]{Muk}) and hence the rank index is at most $6$. To find the exact value of the rank index of these varieties, we develop a basic theory to calculate the lower bound of the rank index using the symmetric matrix associated to a basis of the $\mathbb{K}$-vector space of quadratic defining equations. For details, see Section $2.1$.

Theorem \ref{thm:RNS} says that the rank index of a smooth rational normal scroll $X$ is equal to $3$ if $X$ is a curve and $4$ otherwise. Theorem \ref{thm:Segre} says that the rank index of a Segre variety is equal to $4$. Theorem \ref{thm:Grassmannian} says that the rank index of the Pl\"{u}cker embedding of the Grassmannian of lines is equal to $6$.

In section 6, we investigate the rank index of del Pezzo varieties. Comparing with various well-known and nice characterizations of del Pezzo varieties, we prove that the rank indices of them are different in each case.\\

The paper is structured as follows. In $\S$ $2$, we prove some basic properties of the rank index. From $\S$ $3$ to $\S$ $6$, we study the rank indices of rational normal scrolls, Segre varieties, Pl\"{u}cker embedding of Grassmannian of lines and del Pezzo varieties, in turn. \\

\noindent {\bf Acknowledgement.} The first named author was supported by a KIAS Individual Grant(MG083101) at Korea Institute for Advanced Study. The second named author was supported by the National Research Foundation of Korea(NRF) grant funded by the Korea government(MSIT) (No. 2022R1A2C1002784).

\section{Preliminary}\label{Preliminary}
\noindent Let $X \subset \P^r$ be a nondegenerate irreducible projective variety whose homogeneous ideal $I(X)$ in  the homogeneous coordinate ring $R = \mathbb{K} [z_0 , z_1 , \ldots , z_r ]$ of $\P^r$ is generated by quadratic polynomials. Let $\{ Q_0 , \ldots , Q_t \}$ be a basis of $I(X)_2$. In this section we will explain how to compute the rank index of $X$ from $Q_0 , \ldots , Q_t$. Also we will show a basic fact about the behavior of the rank index under taking a general hyperplane section.

\subsection{Computation of the rank index} Let $X \subset \P^r$ and $\{ Q_0 , \ldots , Q_t \}$ be as above. A general member $Q$ of $I(X)_2$ is of the form
\begin{equation*}
Q = x_0 Q_0 + \cdots + x_t Q_t \quad \mbox{for some} \quad x_0 , \ldots , x_t \in \mathbb{K}.
\end{equation*}
We identify $\P (I(X)_2 )$ with $\P^t$ by sending $[Q]$ to $[x_0 , \ldots , x_t ]$. Now, let $M$ be the $(r+1) \times (r+1)$ symmetric matrix associated with $Q$. Its entries are linear forms in $x_0 , \ldots , x_t$. For each $1 \leq k \leq r$, we define the locus
\begin{equation*}
\Phi_k (X) := \{ [Q] ~|~ Q \in I(X)_2 - \{0\}, ~\mbox {rank} (Q) \leq k \} \subset \P ( I(X)_2 ) = \P^t
\end{equation*}
of all quadratic polynomials of rank $\leq k$ in the projective space $\P^t$. Note that $\Phi_k (X)$ is a projective algebraic set in $\P^t$ since it is the zero set of the ideal $I(k+1,M)$ of $\mathbb{K}[x_0 , \ldots , x_t ]$ generated by all $(k+1) \times (k+1)$ minors of $M$. There is a descending filtration of $\P^t$ associated to $X$:
\begin{equation*}
\emptyset = \Phi_1 (X) = \Phi_2 (X) \subset \Phi_3 (X) \subset \Phi_4 (X) \subset \cdots \subset \Phi_r (X) \subset \P^t
\end{equation*}
Here, $\Phi_1 (X)$ and $\Phi_2 (X)$ are empty since $X$ is irreducible and nondegenerate in $\P^r$. Furthermore, any change of coordinates $T : \P^r \rightarrow \P^r$ induces a projective equivalence between $\Phi_k (X)$ and $\Phi_k (T(X))$.

\begin{proposition}\label{prop:computation of the rank index}
Keep the previous notations. Then
$$\emph{rank-index}(X) = \mbox{min} \{s ~|~\sqrt{I(s+1,M)} \quad \mbox{has no nonzero linear form.} \}$$
\end{proposition}

\begin{proof}
We know already that
\begin{equation*}
V (I(s+1,M)) = \Phi_s (X) \quad \mbox{and} \quad I(\Phi_s (X)) = \sqrt{I(s+1,M)}.
\end{equation*}
Moreover, the following statements are equivalent:\\

\begin{enumerate}
\item[$(i)$] $I(X)$ is generated by quadratic polynomials of rank at most $s$;
\item[$(ii)$] $\Phi_s (X)$ spans $\P (I(X)_2 ) = \P^t$;
\item[$(iii)$] $I(\Phi_s (X)) = \sqrt{I(s+1,M)}$ has no nonzero linear form.\\
\end{enumerate}

\noindent Therefore the rank index of $X$ is equal to the smallest $s$ such that $\sqrt{I(s+1,M)}$ has no nonzero linear forms.
\end{proof}

\subsection{The rank index of a general hyperplane section} Next we show the following elementary but useful fact about the behavior of the rank index under taking a general hyperplane section.

\begin{proposition}\label{prop:hyperplane section}
Let $X \subset \P^r$ be such that ${\mbox depth} (X) \geq 2$ and $I(X)$ is generated by quadratic polynomials. If $Y  \subset \P^{r-1}$ is a general hyperplane section of $X$, then
\begin{equation*}
\emph{rank-index}(X) \leq \emph{rank-index}(Y)+2.
\end{equation*}
\end{proposition}

\begin{proof}
For the simplicity, we may assume that $Y = X \cap V(z_r )$. From the depth condition of $X$, we know that
\begin{equation*}
I(Y) = \left( I(X)+\langle z_r \rangle \right) / \langle z_r \rangle \quad \mbox{in} \quad \mathbb{K} [z_0 , z_1 , \ldots , z_r ] / \langle z_r \rangle = \mathbb{K} [z_0 , z_1 , \ldots , z_{r-1} ].
\end{equation*}
Let $s$ denote the rank index of $Y$. Now, let $Q_0 , Q_1 , \ldots , Q_t$ be quadratic generators of $I(X)$. Then the set $\{ Q_i ' = Q_i +\langle z_r \rangle ~|~ 0 \leq i \leq t \}$ generates $I(Y)$. So, we may assume that the rank of $Q_i '$ is at most $s$ for all $0 \leq i \leq t$. Now, regarding $\mathbb{K} [z_0 , z_1 , \ldots , z_{r-1} ]$ as a subset of $\mathbb{K} [z_0 , z_1 , \ldots , z_{r} ]$, it holds that
$$Q_i = Q_i ' + z_r h_i \quad \mbox{for some linear form} \quad h_i \in \mathbb{K} [z_0 , z_1 , \ldots , z_{r} ].$$
This shows that
$$\mbox{rank}(Q_i ) \leq \mbox{rank}(Q_i ' ) + 2 \leq s+2$$
and hence the rank index of $X$ is less than or equal to $\mbox{rank-index}(Y)+2$.
\end{proof}

\begin{remark}
In Theorem \ref{thm:non-normal del Pezzo}, we show that there is a surface $X \subset \P^r$ with a general hyperplane section $\mathcal{C} \subset \P^{r-1}$ such that
$$\mbox{rank-index}(X) = 5 \quad \mbox{and} \quad \mbox{rank-index}(\mathcal{C}) = 3,$$
and that there is a threefold $X \subset \P^r$ with a general hyperplane section $S \subset \P^{r-1}$ such that
$$\mbox{rank-index}(X) = 6 \quad \mbox{and} \quad \mbox{rank-index}(S) = 4 .$$
Thus the upper bound of the rank index of $X$ in  Proposition \ref{prop:hyperplane section} cannot be improved.
\end{remark}

\subsection{The construction of rank $3$ quadratic equations}
Let $L = \mathcal{O}_X (1)$ and suppose that $X \subset \P^r$ is the linearly normal embedding of $X$ by $L$. Thus there is a natural isomorphism
$$f : H^0 (X,L) \rightarrow R_1$$
of $\mathbb{K}$-vector spaces. Now, assume that $L$ is decomposed as $L=A ^{2} \otimes B$ for some line bundles $A$ and $B$ on $X$ such that $h^0 (X,A) \geq 2$ and $h^0 (X,B) \geq 1$. Define the map
\begin{equation*}
Q_{A,B} : H^0 (X,A) \times H^0 (X,A) \times H^0 (X,B) \rightarrow I(X)_2
\end{equation*}
by
\begin{equation*}
Q_{A,B} (s,t,h) = f(s \otimes s \otimes h) f(t \otimes t \otimes h) - f(s
\otimes t \otimes h )^2 .
\end{equation*}
This map is well-defined since the restriction of $Q_{A,B} (s,t,h)$ to $X$ becomes
\begin{equation*}
(s\otimes s\otimes h)|_X \times(t\otimes t\otimes h)|_X - (s\otimes
t\otimes h)|_X ^2 =0.
\end{equation*}
Thus, $Q_{A,B} (s,t,h)$ is either $0$ or else a rank $3$ quadratic equation of $X$.

\begin{example}\label{ex:2-uple of P3}
Let $X = \nu_2 (\P^3 ) \subset \P^9$ and consider the $Q$-map for the decomposition $\mathcal{O}_{\P^3} (3) = \mathcal{O}_{\P^3} (1)^2 \otimes \mathcal{O}_{\P^3}$. That is,
\begin{equation*}
Q : H^0 (X,\mathcal{O}_{\P^3} (1) ) \times H^0 (X,\mathcal{O}_{\P^3} (1) ) \times H^0
(X,\mathcal{O}_{\P^3} ) \rightarrow I(X )_2 .
\end{equation*}
Let $\{x_0,x_1,x_2 ,x_3 \}$ be a basis of $H^0 (X,\mathcal{O}_{\P^3} (1) )$. Then we can produce the following three finite sets of rank $3$ quadratic polynomials in $I(X )$:
\begin{align*}
\Gamma_{11}&=\{Q(x_i,x_j,1) ~|~ 0\leq i <j\leq 3\}\\
\Gamma_{12}&=\{Q(x_i +x_j , x_k ,1) ~|~ 0\leq i<j\leq 3, ~0 \leq k \leq 3 ,~ k \neq i,j \}\\
\Gamma_{22}&=\{Q(x_i +x_j , x_k +x_l ,1) ~|~ \{i,j,k,l \} = \{0 ,1,2,3 \}\}
\end{align*}
Then $\Gamma := \Gamma_{11} \cup \Gamma_{12} \cup \Gamma_{22}$ consists of $21$ quadratic polynomials of rank $3$. Due to the proof of \cite[Theorem 3.1]{HLMP}, it holds that if $\mbox{char} (\mathbb{K}) \neq 2,3$ then $I(X)$ is generated by $\Gamma$ and hence $X \subset \P^9$ satisfies property $QR(3)$, while if $\mbox{char} (\mathbb{K}) = 3$ then $X \subset \P^9$ fails to satisfy property $QR(3)$.
\end{example}

\section{The rank index of rational normal scrolls}
\noindent This section is devoted to the proof of Theorem \ref{thm:RNS} below.

\begin{theorem}\label{thm:RNS}
Let $X \subset \P^r$ be an $n$-dimensional smooth rational normal scroll. Then
\begin{equation*}
\emph{rank-index}(X) = \begin{cases} 3 & \mbox{if} \quad  n=1, \quad \mbox{and}\\
                                     4 & \mbox{if} \quad n \geq 2. \end{cases}
\end{equation*}
\end{theorem}

\begin{proof}
It is well-known that $I(X)$ is generated by $2$-minors of a $1$-generic matrix. This implies that the rank index of $X$ is at most $4$.

If $n=1$ and hence $X \subset \P^r$ is a rational normal curve, then the rank index of $X$ is equal to $3$ by \cite[Corollary 2.4]{HLMP}.

Suppose that $n = 2$ and hence $X = S(a,b) \subset \P^{a+b+1}$ for some positive integers $a$ and $b$. Thus we have
\begin{equation*}
X = \{ [s^a x :\cdots:s^{a-i} t^i x :\cdots:t^a x : s^b y : \cdots : s^{b-j} t^j y : \cdots:t^b y ]~|~[s:t],[x:y] \in \P^1 \}.
\end{equation*}
Let $x_0 , x_1 , \ldots , x_a , y_0 , y_1 , \ldots , y_b$ be the homogeneous coordinates of $\P^{a+b+1}$. Then $I(X)$ is generated by the $2$-minors of the matrix
\begin{equation*}
\begin{bmatrix}
x_0 & x_1 & \cdots & x_{a-1} & y_0 & y_1 & \cdots & y_{b-1} \\
x_1 & x_2 & \cdots & x_a     & y_1 & y_2 & \cdots & y_{b}
\end{bmatrix}.
\end{equation*}
More precisely, $I(X)$ is generated by $A \cup B \cup C$ where
\begin{equation*}
\begin{cases}
A = \{ F_{ij} := x_i x_{j+1} - x_{i+1} x_j ~|~ 0 \leq i < j \leq a-1 \},\\
B = \{ G_{ij} := y_i y_{j+1} - y_{i+1} y_j ~|~ 0 \leq i < j \leq b-1 \} \quad \mbox{and} \\
C = \{ H_{ij} := x_i y_{j+1} - x_{i+1} y_j ~|~ 0 \leq i \leq a-1 , 0 \leq j \leq b-1 \}.
\end{cases}
\end{equation*}
Now, let
\begin{equation*}
Q = \sum  \alpha_{ij} F_{ij} + \sum \beta_{ij} G_{ij} + \sum \gamma_{ij} F_{ij}
\end{equation*}
be a general element of $I(X)_2$. Write $Q$ as
\begin{equation*}
Q = \sum  a_{ij} x_i x_j + \sum b_{ij} y_i y_j + \sum c_{ij} x_i y_j .
\end{equation*}
Then it holds that
\begin{equation*}
a_{00} = a_{01} = b_{00} = b_{01} = c_{00} = 0, \quad c_{01} = \gamma_{00} \quad \mbox{and} \quad c_{10} = -\gamma_{00} .
\end{equation*}
The symmetric matrix $M_Q$ associated to $Q$ has the following $4 \times 4$ submatrix.
\begin{equation*}
N = \begin{bmatrix}
2a_{00}  &  a_{01} & c_{00}  &  c_{01} \\
a_{01}   & 2a_{11} &  c_{10}  & c_{11} \\
c_{00}   & c_{10}  & 2b_{00}  &  b_{01} \\
c_{01}   & c_{11}  & b_{01}  &  2b_{11}
\end{bmatrix} =
 \begin{bmatrix}
0  & 0  &  0 & \gamma_{00} \\
0 & 2a_{11}   & -\gamma_{00}  & c_{11} \\
0  & -\gamma_{00}   &  0 &  0 \\
\gamma_{00}    & c_{11}  & 0    & 2b_{11}
\end{bmatrix}
\end{equation*}
Therefore $\det (N) = \gamma_{00} ^4 \in I(4,M)$ and hence $\gamma_{00} \in \sqrt{I(4,M)}$, which shows that the smooth rational normal surface scroll $X \subset \P^{a+b+1}$ fails to satisfy property $QR(3)$. This shows that the rank index of $X$ is equal to $4$.

Finally, suppose that $n \geq 3$. Then a general surface section $S$ of $X$ is a smooth rational normal surface scroll. If $X$ satisfies property $QR(3)$, then so does $S$ since $X$ is arithmetically Cohen-Macaulay. Therefore $X$ fails to satisfy property $QR(3)$, which completes the proof that the rank index of $X$ is $4$.
\end{proof}

\begin{remark}
$(1)$ Let $X \subset \P^r$ be a nondegenerate irreducible projective variety of dimension $n$, codimension $e$ and degree $d$. Then it holds always that
$d \geq e+1$. Also $X$ is called a variety of minimal degree if $d = e+1$. Varieties of minimal degree were completely classified more than hundred years ago by P. del Pezzo and E. Bertini (cf. \cite{EHa}). If $e \geq 2$, then $X$ is a variety of minimal degree if and only if it is (a cone over) the Veronese surface in $\P^5$ or (a cone over) a smooth rational normal scroll.

\noindent $(2)$ Combining Theorem \ref{thm:RNS} with \cite[Theorem 1.1 and 1.2]{HLMP} gives us the rank index of varieties of minimal degree perfectly. In particular, if $X \subset \P^r$ is a smooth variety of minimal degree and of codimension $\geq 2$ then its rank index is equal to $3$ or $4$. Also it is equal to $3$ if and only if $X$ is a Veronese variety, or equivalently, if and only if $X$ is either a rational normal curve or else the Veronese surface in $\P^5$.
\end{remark}

\section{The rank index of Segre varieties}
\noindent This section is devoted to giving a proof of Theorem \ref{thm:Segre} below.

\begin{theorem}\label{thm:Segre}
Let $X = \P^{n_1} \times \cdots \times \P^{n_{\ell}} \subset \P^r$, $r+1 = (n_1 +1) \cdots (n_{\ell} +1)$, be the Segre variety for some $\ell \geq 2$. Then
\begin{equation*}
\emph{rank-index}(X) = 4.
\end{equation*}
Furthermore, $I(X)$ has no quadratic polynomials of rank $3$.
\end{theorem}

\begin{proof}
In \cite[Theorem 1.5]{Ha} it is proved that the ideal of the Segre variety $X$ is generated by all the $2$-minors of a generic
hypermatrix of indeterminates. This implies that the rank index of $X$ is at most $4$. Therefore the proof is completed if $X$ satisfies no quadratic polynomial of rank $3$.

Now, suppose that $I(X)$ contains an element $Q$ of rank $3$. Then $Q = xy-z^2$ for some linear forms $x,y,z$ on $\P^r$. Write the Weil divisors $\mbox{div}(z)$, $\mbox{div}(x)$ and $\mbox{div}(y)$ as
\begin{equation*}
\mbox{div}(z) = \sum_{i=1} ^t  d_i D_i ,\quad  \mbox{div}(x) = \sum_{i=1} ^t e_i D_i  \quad \mbox{and} \quad  \mbox{div}(y) = \sum_{i=1} ^t f_i D_i
\end{equation*}
where $D_1 , \ldots , D_t$ are prime divisors on $X$. Suppose that $e_1 , \ldots , e_s$ are odd and $e_{s+1} , \ldots , e_t$ are even. Since $xy = z^2$, it follows that $f_i$ is odd if and only if $1\leq i \leq s$. Therefore
\begin{equation*}
\mathcal{O}_X (1) \otimes \mathcal{O}_X (-D_1 - \cdots - D_s ) = A^2
\end{equation*}
where $H^0 (X,A)$ contains two sections, say $\alpha$ and $\beta$, such that
\begin{equation*}
\mbox{div}(x) - (D_1 + \cdots + D_s ) = \mbox{div}(\alpha^2) \quad \mbox{and} \quad  \mbox{div}(y) - (D_1 + \cdots + D_s ) = \mbox{div}(\beta^2) .
\end{equation*}
It follows that
\begin{equation*}
\mbox{div}(\alpha) = \sum_{i=1} ^{s}  \frac{e_i -1}{2} D_i + \sum_{i=s+1} ^m  \frac{e_i}{2} D_i \quad \mbox{and} \quad \mbox{div}(\beta) =  \sum_{i=1} ^{s}  \frac{f_i -1}{2} D_i + \sum_{i=\ell+1} ^t  \frac{f_i}{2} D_i .
\end{equation*}
If $\alpha$ and $\beta$ are $\mathbb{K}$-linearly dependent then $\mbox{div}(\alpha) = \mbox{div}(\beta)$ and hence $\mbox{div}(x) = \mbox{div}(y)$. This means that $x$ and $y$ are $\mathbb{K}$-linearly dependent and hence $Q = z^2 -xy$ is not irreducible, which is a contradiction. Therefore $\alpha$ and $\beta$ are $\mathbb{K}$-linearly independent and hence $h^0 (X,A) \geq 2$. Now, let $B$ denote the line bundle $\mathcal{O}_X ( D_1 + \cdots + D_s )$ and let $h \in H^0 (X,B)$ be the element corresponding to the divisor $D_1 + \cdots + D_s$. Then we have
\begin{equation*}
\mathcal{O}_X (1) = A^2 \otimes B \quad \mbox{and} \quad Q = Q_{A,B} (\alpha,\beta,h).
\end{equation*}
But $\mathcal{O}_X (1)$ is the line bundle defining the Segre embedding of $\P^{n_1} \times \cdots \times \P^{n_{\ell}}$ and hence it does not admit a factorization of the form $\mathcal{O}_X (1) = A^2 \otimes B$ such that $h^0 (X,A) \geq 2$ and $h^0 (X,B) \geq 1$. This completes the proof that $X$ satisfies no quadratic polynomial of rank $3$.
\end{proof}

\section{The rank index of the Pl\"{u}cker embedding of $\mathbb{G}(1,\P^n )$}
\noindent This section is devoted to proving Theorem \ref{thm:Grassmannian} below.

\begin{theorem}\label{thm:Grassmannian}
Let $X = \mathbb{G} (1,\P^n ) \subset \P^{{n+1 \choose 2}-1}$, $n \geq 3$, be the Pl\"{u}cker embedding of the Grassmannian manifold of lines in $\P^n$. Then
\begin{equation*}
\emph{rank-index}(X) = 6.
\end{equation*}
Furthermore, $I(X)$ has no quadratic polynomials of rank $\leq 5$.
\end{theorem}

\begin{proof}
Let $\{p_{ij} ~|~ 0 \le i < j \le n \}$ be the set of homogeneous coordinates of the projective space $\P^{{n \choose 2}-1}$. Then the homogeneous ideal of $X$ is generated by the following Pl\"{u}cker relations
\begin{equation*}
Q(i,j,k,l) := p_{ij} p_{kl} - p_{ik} p_{jl} + p_{jk} p_{il} ,\quad 0 \le i < j < k < l \le n
\end{equation*}
(cf. \cite[Theorem 8.15]{Muk}). Therefore the rank index of $X$ is at most $6$. Also the proof is completed if $X$ satisfies no quadratic polynomial of rank $\leq 5$.

Now, suppose that
\begin{equation*}
Q = \sum_{0 \le i < j < k < l \le n} x_{i,j,k,l} Q(i,j,k,l)
\end{equation*}
is of rank at most $5$. Then all $6$-minors of the symmetric matrix $M_Q$ associated to $Q$ must vanish. Fix $(i,j,k,l)$ with $0 \le i < j < k < l \le n$. Then the $6 \times 6$ submatrix $N$ of $M_Q$ corresponding to the ordered linear forms $p_{ij} , p_{ik} , p_{il} , p_{jk} , p_{jl} ,  p_{kl}$ is equal to
$$N = \begin{bmatrix}
0           &  0            & 0            & 0            & 0             &  x_{i,j,k,l}  \\
0           & 0             & 0            & 0            &  -x_{i,j,k,l} & 0   \\
0           & 0             & 0            &  x_{i,j,k,l} & 0             &  0  \\
0           & 0             &  x_{i,j,k,l} & 0            & 0             &  0  \\
0           &  -x_{i,j,k,l} & 0            &    0         & 0             &   0    \\
x_{i,j,k,l} & 0             &    0         & 0            &   0           &    0
\end{bmatrix}$$
Since $\det (N) = - x_{i,j,k,l}^6 $ is equal to zero, it holds that $x_{i,j,k,l} = 0$. This completes the proof that $Q = 0$, and hence $I(X)$ has no quadratic polynomials of rank $\leq 5$.
\end{proof}

\section{The rank index of some del Pezzo varieties}
\noindent This section is devoted to investigating the rank index of del Pezzo varieties. Recall that a projective variety $X \subset \P^r$ of codimension $e$ and degree $d$ is called a del Pezzo variety if it is arithmetically Cohen-Macaulay and $d = e+2$. In \cite{EGHP} and \cite{HK1}, it is shown that varieties of minimal degree and del Pezzo varieties are respectively the extremal and next to extremal varieties with respect to property $N_p$. Also, the authors in \cite{HK2} show that varieties of minimal degree and del Pezzo varieties are respectively the extremal and next to extremal varieties with respect to the upper bounds of Betti numbers in the quadratic strand. Comparing with those nice characterizations of del Pezzo varieties, we prove that the rank indices of them are different in each case.

Let $X \subset \P^{n+e}$ be a del Pezzo variety of dimension $n \geq 1$ and codimension $e \geq 2$ which is not a cone. If $n=1$, then $X$ is a linearly normal curve of arithmetic genus $1$ and degree $e+2 \geq 4$. Thus the rank index of $X$ is always equal to $3$ by \cite[Theorem 1.1]{P}. If $e=2$ then $X$ is a complete intersection of two quadrics and hence its rank index can be any number between $3$ and $n$. From now on, we suppose that $n \geq 2$ and $e \geq 3$. We will discuss the rank index of $X$ by dividing del Pezzo varieties into the following three types:\\

\begin{enumerate}
\item[$(i)$] Smooth del Pezzo varieties;
\item[$(ii)$] Non-normal del Pezzo varieties;
\item[$(iii)$] Singular normal del Pezzo varieties.\\
\end{enumerate}

\noindent For cases $(i)$ and $(ii)$, the classification up to projective equivalence is perfectly known (cf. \cite{F} and \cite{LP}). In the next three subsections, we will look at what the values of the rank index are in these two cases.

\subsection{Smooth del Pezzo surfaces} Suppose that $X$ is a smooth surface. Then $X$ is of one of the following two types:\\

\begin{enumerate}
\item[$(1)$] There is a finite set $\Sigma$ on $\P^2$ of $t$ points ($0 \leq t \leq 4$) in linearly general position such that $X \subset \P^{9-t}$ is obtained by the inner projections of the triple Veronese surface $\nu_3 (\P^2 ) \subset \P^9$. We denote this $X$ by $S_t$. For example, $S_0 = \nu_3 (\P^2 )$.
\item[$(2)$] $X$ is the $2$-uple Veronese embedding of the smooth quadric surface $Q$ in $\P^3$. We denote this $X$ by $S_5$.   \\
\end{enumerate}

\begin{theorem}\label{thm:smooth del Pezzo surface}
Let $S_t$, $0 \leq t \leq 5$, be as above. Then
\begin{equation*}
\emph{rank-index}(S_t) = \begin{cases} 3 & \mbox{if} \quad  t=0,5, \quad \mbox{and}\\
                                     4 & \mbox{if} \quad t = 1,2,3,4 \end{cases}
\end{equation*}
\end{theorem}

\begin{proof}
All computations in the proof are obtained by means of the Computer Algebra System ``MACAULAY 2" \cite{GS}.

The triple Veronese surface $S_0 \subset \P^9$ is given by the following parametrization:
\begin{equation*}
S_0 = \{ [x_0^3 :x_0^2x_1 :x_0^2x_2 :x_0x_1^2 :x_0x_1x_2: x_0x_2^2:x_1^3 : x_1^2x_2 :x_1x_2^2 :x_2^3 ] ~|~[x_0,x_1,x_2] \in \P^2 \}
\end{equation*}
Note that the homogeneous ideal $I(S_0 )$ of $S_0$ in $\mathbb{K} [z_0,\ldots,z_9]$ is generated by $27$ quadrics. Now, consider the $Q$-map introduced in $\S$ 2.3 for the decomposition $\mathcal{O}_{\P^2} (3) = \mathcal{O}_{\P^2} (1)^2 \otimes \mathcal{O}_{\P^2} (1)$. That is,
\begin{equation*}
Q : H^0 (X,\mathcal{O}_{\P^2} (1) ) \times H^0 (X,\mathcal{O}_{\P^2} (1) ) \times H^0
(X,\mathcal{O}_{\P^2} (1) ) \rightarrow I(S_0 )_2 .
\end{equation*}
Let $\{x_0,x_1,x_2\}$ be a basis of $H^0 (X,\mathcal{O}_{\P^2} (1) )$. Then we can produce the following three finite sets of rank $3$ quadratic polynomials in $I(S_0 )$:
\begin{align*}
\Gamma_{111}&=\{Q(x_i,x_j,x_k) ~|~ 0\leq i <j\leq 2, 0\leq k\leq 2\}\\
\Gamma_{112}&=\{Q(x_i,x_j,x_k+x_l) ~|~ 0\leq i<j\leq 2, 0\leq k<l\leq 2\}\\
\Gamma_{121}&=\{Q(x_i,x_j+x_k,x_l) ~|~ 0\leq i\leq 2, 0\leq\j<k\leq 2, j,k\neq i, 0\leq l\leq 2\}
\end{align*}
Note that $|\Gamma_{111}|=|\Gamma_{112}|=|\Gamma_{121}| =9$. Using Macaulay2, we can check that $$\Gamma_{111}\cup\Gamma_{112}\cup\Gamma_{121}$$
is $\mathbb{K}$-linearly independent and hence a basis of $I(S_0 )_2$. This show that the rank index of $S_0$ is equal to $3$.

The following surface $S_1 \subset \mathbb{P}^8$ is obtained by the inner projection of $S_0 = \nu_3 (\P^2 )$ in $\P^9$ from the point $\nu_3 ([0:0:1])$.
\begin{equation*}
S_1 = \{ [x_0^3 :x_0^2 x_1 :x_0^2 x_2 :x_0x_1^2 : x_0x_1x_2 :x_0x_2^2 :x_1^3 :x_1^2 x_2 :x_1x_2^2 ] ~|~[x_0,x_1,x_2] \in \P^2 \}
\end{equation*}
The homogeneous ideal $I(S_1 )$ of $S_1$ in $\mathbb{K}[z_0 , z_1 , \ldots ,z_8]$ is generated by the following set of $20$ $\mathbb{K}$-linearly independent quadratic polynomials:
\begin{equation*}
\{Q_0,Q_1,\ldots,Q_{19}\}=
\end{equation*}
\begin{equation*}
\{z_7 ^2 - z_6 z_8 , ~z_5 z_7 - z_4 z_8 , ~z_4 z_7 -z_3 z_8 ,~z_2 z_7 - z_1 z_8 ,~z_5 z_6 - z_3 z_8 ,~z_4 z_6 -z_3 z_7 ,~ z_2 z_6 - z_1 z_7 ,
\end{equation*}
\begin{equation*}
~\bold{z_4 z_5 - z_2 z_8} ,~ z_3 z_5 - z_1 z_8 , ~z_1 z_5 - z_0 z_8 , ~z_4 ^2 - z_1 z_8 , ~z_3 z_4 - z_1 z_7 ,~ z_2 z_4 - z_0 z_8 , ~z_1 z_4 - z_0 z_7 ,
\end{equation*}
\begin{equation*}
z_3 ^2 - z_1 z_6 , ~z_2 z_3 - z_0 z_7 , ~z_1 z_3 - z_0 z_6 ,~ z_2 ^2 - z_0 z_5 ,~ z_1 z_2 - z_0 z_4 , ~z_1 ^2 - z_0 z_3\}  \quad \quad \quad \quad  \quad \quad \quad
\end{equation*}
From the above generating set of $I(S_1)$, it follows that the rank index of $S_1$ is at most $4$. From now on, we will show that $\mbox{rank-index}(S_1 ) > 3$. A nonzero quadratic polynomial $Q$ in $I(S_1 )$ is of the form
\begin{equation*}
Q = \sum_{i=0}^{19} a_i Q_i = \sum_{0 \leq i < j \leq 8} \alpha_{i,j} z_i z_j .
\end{equation*}
Thus it holds that
\begin{equation*}
\alpha_{5,5} = \alpha_{5,8} = \alpha_{2,5} = \alpha_{8,8} = 0 \quad \mbox{and} \quad \alpha_{4,5} = - \alpha_{2,8} = a_7 .
\end{equation*}
Let $M$ be the $9 \times 9$ symmetric matrix associated to $Q$. Thus $M_{i,j} = \alpha_{i,j}$ if $i <j$ and $M_{i,i} = 2 \alpha_{i,i}$. Now, consider the following $4 \times 4$ sub-matrix $N$ of $M$:
$$N = \begin{bmatrix}
2 \alpha_{5,5} & \alpha_{5,8}   & \alpha_{2,5}   & \alpha_{4,5}    \\
\alpha_{5,8}   & 2 \alpha_{8,8} & \alpha_{2,8}   & \alpha_{4,8}   \\
\alpha_{2,5}   & \alpha_{2,8}   & 2 \alpha_{2,2} & \alpha_{2,4}    \\
\alpha_{4,5}   & \alpha_{4,8}   & \alpha_{2,4}   & 2 \alpha_{4,4}
\end{bmatrix} = \begin{bmatrix}
0              & 0              & 0              & a_7    \\
0              & 0              & -a_7             & \alpha_{4,8}   \\
0              & -a_7             & 2 \alpha_{2,2} & \alpha_{2,4}    \\
a_7              & \alpha_{4,8}   & \alpha_{2,4}   & 2 \alpha_{4,4}
\end{bmatrix}$$
Then
\begin{equation*}
\det (N) = a_7^4 \in I(4,M) \quad \mbox{and hence} \quad a_7 \in \sqrt{I(4,M)} = I(\Phi_3 (S_2 )).
\end{equation*}
Hence the rank-index of $S_1$ is at least $4$ by Proposition \ref{prop:computation of the rank index}. This completes the proof that the rank index of $S_1$ is equal to $4$.

Let $p_1 = \nu_3 ([0:0:1])$, $p_2 = \nu_3 ([0:1:0])$, $p_3 = \nu_3 ([1:0:0])$ and $p_4 = \nu_3 ([1,1,1])$. For each $1 \leq t \leq 4$, we may assume that $S_t$ is the inner projection of $S_0$ from $p_1,\ldots,p_t$. The homogeneous ideal $I(S_2)$ is generated by the $14$ quadratic polynomials among $\{Q_0,\ldots,Q_{19}\}$ without $z_6$ in its support. Since $z_4z_5-z_2z_8\in I(S_2)\subset I(S_1)$, the above argument is still working. Hence the rank-index of $S_2$ is $4$. Similarly, the homogeneous ideal $I(S_3)$ is generated by $9$ quadratic polynomials among $\{Q_0,\ldots,Q_{19}\}$ without $z_0$ and $z_6$ in its support. Hence $z_4z_5-z_2z_8\in I(S_3)\subset I(S_2)\subset I(S_1)$ and the above argument is still working.

For $t=4$, we need more argument. The homogeneous ideal $I(S_4)$ is generated by $5$ quadratic polynomials, say $Q_0',\ldots,Q_4'$ where
$$Q_0 ' = (z_3-z_4)z_5-(z_2-z_4)z_7,~ Q_1 ' = (z_4-z_5)z_4-(z_2-z_5)z_7, \quad \quad \quad \quad \quad \quad$$
$$Q_2 ' = (z_3-z_5)z_4-(z_1+z_2-z_4-z_5)z_7,~ Q_3 ' = (z_4-z_7)z_2-(z_1-z_5)z_7 ~ \mbox{and} $$
$$Q_4 ' = (z_3-z_7)z_2-(z_1-z_7)z_4 . \quad \quad \quad \quad \quad \quad \quad \quad \quad \quad \quad \quad \quad \quad \quad \quad \quad \quad \quad \quad$$
From the above generating set of $I(S_1)$, it follows that the rank index of $S_4$ is at most $4$.
\begin{equation*}
Q' = \sum_{i=0}^{4} a_i Q_i' = \sum \alpha_{i,j} z_i z_j .
\end{equation*}
Now, consider the following $4 \times 4$ sub-matrix $N$ of $M$:
$$N = \begin{bmatrix}
2 \alpha_{1,1} & \alpha_{1,2}   & \alpha_{1,3}   & \alpha_{1,4}    \\
\alpha_{1,2}   & 2 \alpha_{2,2} & \alpha_{2,3}   & \alpha_{2,4}   \\
\alpha_{1,3}   & \alpha_{2,3}   & 2 \alpha_{3,3} & \alpha_{3,4}    \\
\alpha_{1,4}   & \alpha_{2,4}   & \alpha_{3,4}   & 2 \alpha_{4,4}
\end{bmatrix} = \begin{bmatrix}
0              & 0              & 0              & -a_4    \\
0              & 0              & a_4            & \alpha_{2,4}   \\
0              & a_4             & 2 \alpha_{3,3} & \alpha_{3,4}    \\
-a_4          & \alpha_{2,4}   & \alpha_{3,4}   & 2 \alpha_{4,4}
\end{bmatrix}$$
Then
\begin{equation*}
\det (N) = a_4^4 \in I(4,M_{Q'}) \quad \mbox{and hence} \quad a_4 \in \sqrt{I(4,M_{Q'})} = I(\Phi_3 (S_4 )).
\end{equation*}
Hence the rank-index of $S_4$ is at least $4$ by Proposition \ref{prop:computation of the rank index}. This completes the proof that the rank index of $S_4$ is equal to $4$

Finally, we consider $S_5$ which is equal to $\nu_2 (Q)$ where $Q \subset \P^3$ is the smooth quadric surface $V(x_0 x_3 - x_1 x_2)$. Note that $S_5$ is hyperplane section of
\begin{equation*}
Y = \nu_2(\mathbb{P}^3) = \{ [x_0^2 :x_0 x_1 :x_0 x_2 :x_0 x_3: x_1^2 :x_1x_2 :x_1x_3 : x_2^2 : x_2x_3 :x_3^2 ] ~|~[x_0,x_1,x_2] \in \P^2 \}
\end{equation*}
by $z_3 - z_5$ where $z_0 , z_1 , \ldots , z_9$ are the homogeneous coordinates of $\P^9$. Let $Q'$ be the restriction of $z_3 z_5 - z_2 z_6$ in $I(Y)$ to $\P^8 = V(z_3 -z_5 )$. Thus $Q'$ is of rank $3$ and contained in $I(S_5 )$. Now, let $\Gamma '$ be the restriction of $\Gamma$ in Example \ref{ex:2-uple of P3}. Then one can show that the $\mathbb{K}$-vector space spanned by $\Gamma' \cup \{Q' \}$ is of dimension $20$. This implies that $I(S_5 )$ is generated by $\Gamma' \cup \{Q' \}$ and hence $S_5$ satisfies property $\textsf{QR}(3)$.
\end{proof}

\begin{remark}
Suppose that $\mbox{char}(\mathbb{K}) \neq 2,3$. Then the assertion that $\mbox{rank-index}(S_t) = 3$ for $t = 0,5$ comes directly from \cite[Theorem 1.1 and Corollary 5.2]{HLMP}.
\end{remark}

\subsection{Smooth del Pezzo varieties of dimension $\geq 3$}
Suppose that $X$ is a del Pezzo manifold of dimension $n \geq 3$. Then $X$ is of one of the following types:\\

\begin{enumerate}
\item[$(1)$] ($d=8$) $X$ is the second Veronese variety $\nu_2 (\P^3 ) \subset \P^9$;
\item[$(2)$] ($d=7$) $X$ is the inner projection of $\nu_2 (\P^3 ) \subset \P^9$;
\item[$(3)$] ($d=6$) $X = \P^2 \times \P^2 \subset \P^8$ or $\P_{\P^2} (\mathcal{T}_{\P^2}) \subset \P^7$ or $\P^1 \times \P^1 \times \P^1 \subset \P^7$;
\item[$(4)$] ($d=5$) $X$ is a linear section of $\mathbb{G} (1,\mathbb{P}^4) \subset \P^9$.\\
\end{enumerate}
For details, we refer the reader to see \cite[Section 8]{F}. From now on, we study the rank index of $X$ for the above four cases in turn.

\renewcommand{\descriptionlabel}[1]%
             {\hspace{\labelsep}\textrm{#1}}
\begin{description}
\setlength{\labelwidth}{13mm}
\setlength{\labelsep}{1.5mm}
\setlength{\itemindent}{0mm}

\item[$(1)$] For $X = \nu_2 (\P^3 ) \subset \P^9$, it holds that
\begin{equation*}
\mbox{rank-index}(X) = \begin{cases} 3 & \mbox{if} \quad  \mbox{char}(\mathbb{K}) \neq 2,3, \quad \mbox{and}\\
                                     4 & \mbox{if} \quad  \mbox{char}(\mathbb{K}) = 3. \end{cases}
\end{equation*}
For the proof, see \cite[Theorem 1.1 and 1.2]{HLMP}.
\smallskip

\item[$(2)$] Suppose that $X$ is the inner projection of $\nu_2 (\P^3 ) \subset \P^9$ from a point $p \in \nu_2 (\P^3 )$. We may assume that $p = \nu_2 ([0:0:1])$ and hence $X$ is as follows:
\begin{equation*}
X = \{[x^2 :xy:xz:xw:y^2 : yz:yw:z^2 :zw] ~|~ [x:y:z:w] \in \P^3 \} \subset \P^8
\end{equation*}
Then $I(X)$ is generated by the following $14$ quadratic polynomials:
\begin{equation*}
z_6 z_7 - z_5 z_8 ,~z_3 z_7 - z_2 z_8 ,~z_5 z_6 - z_4 z_8 ,~z_2 z_6 - z_1 z_8 ,~ z_5 ^2 - z_4 z_7 ,
\end{equation*}
\begin{equation*}
z_3 z_5 - z_1 z_8 ,~z_2 z_5 - z_1 z_7 ,~z_3 z_4 - z_1 z_6 ,~ z_2 z_4 - z_1 z_5 ,~z_2 z_3 - z_0 z_8 ,
\end{equation*}
\begin{equation*}
 z_1 z_3 - z_0 z_6 , ~z_2 ^2 - z_0 z_7 ,~ z_1 z_2 - z_0 z_5 , ~ z_1 ^2 - z_0 z_4 \quad \quad \quad  \quad \quad \quad  \quad
\end{equation*}
Then one can prove that $\mbox{rank-index}(X) = 4$ in the same way as in the proof of Theorem \ref{thm:smooth del Pezzo surface} for $S_1$.
\smallskip

\item[$(3)$] If $X$ is $\P^2 \times \P^2$ or $\P^1 \times \P^1 \times \P^1$, then $\mbox{rank-index}(X) = 4$ by Theorem \ref{thm:Segre}. Also $X = \P_{\P^2} (\mathcal{T}_{\P^2}) \subset \P^7$ is obtained as a general hyperplane section of the Segre variety $\P^2 \times \P^2 \subset \P^8$. More precisely, we may assume that $X = \P^2 \times \P^2 \cap V(z_0 +z_4 +z_8 )$ and so $I(X)$ in $\mathbb{K} [z_0 ,z_1 , \ldots , z_7]$ is generated by the following $9$ quadratic polynomials, say $Q_1 , \ldots , Q_9$:
\begin{equation*}
\quad \quad \quad z_0 z_4 + z_4 ^2 + z_5 z_7 ,~z_0 ^2 +z_0 z_4 + z_2 z_6 ,~z_0 z_1 + z_1 z_4 + z_2 z_7 ,~z_0 z_3 + z_3 z_4 + z_5 z_6 ,
\end{equation*}
\begin{equation*}
z_4 z_6 - z_3 z_7 ,~z_1 z_6 - z_0 z_7 , ~z_2 z_4 - z_1 z_5 , ~z_2 z_3 - z_0 z_5 ,~ z_1 z_3 - z_0 z_4 \quad \quad
\end{equation*}
Since $\P^2 \times \P^2 \subset \P^8$ is arithmetically Cohen-Macaulay and of rank index $4$, it follows that the rank index of $X$ is at most $4$. Now, let $M$ be the $8 \times 8$ symmetric matrix associated to $Q = x_1 Q_1 + \cdots + x_9 Q_9$, and write
\begin{equation*}
Q = \sum_{0 \le i \le j \le 7} \alpha_{i,j} z_i z_j .
\end{equation*}
Then $M$ has the following $4 \times 4$ sub-matrix $N$:
$$N = \begin{bmatrix}
2 \alpha_{3,3} & \alpha_{3,4}   & \alpha_{3,6}   & \alpha_{3,7}    \\
\alpha_{3,4}   & 2 \alpha_{4,4} & \alpha_{4,6}   & \alpha_{4,7}   \\
\alpha_{3,6}   & \alpha_{4,6}   & 2 \alpha_{6,6} & \alpha_{6,7}    \\
\alpha_{3,7}   & \alpha_{4,7}   & \alpha_{6,7}   & 2 \alpha_{7,7}
\end{bmatrix} = \begin{bmatrix}
0              & x_4              & 0              & -x_5    \\
x_4            & 2x_1             & x_5            & 0   \\
0              & x_5              & 0              & 0    \\
-x_5           & 0                & 0              & 0
\end{bmatrix}$$
In particular, we have
\begin{equation*}
\det (N) = x_5 ^4 \in I(4,M) \quad \mbox{and hence} \quad x_5 \in \sqrt{I(4,M)} = I(\Phi_3 (X )).
\end{equation*}
Hence the rank-index of $X$ is at least $4$ by Proposition \ref{prop:computation of the rank index}. This completes the proof that the rank index of $X$ is equal to $4$.
\smallskip

\item[$(4)$] If $X = \mathbb{G} (1,\mathbb{P}^4) \subset \P^9$, then $\mbox{rank-index}(X) = 6$ by Theorem \ref{thm:Grassmannian}. Now, let
$$X_k \subset \P^{k+3} \quad (3 \leq k \leq 5)$$
be a smooth $k$-dimensional linear section of $\mathbb{G} (1,\mathbb{P}^4) \subset \P^9$. Note that $X_k$ is unique up to projective equivalence for each $3 \leq k \leq 5$ (cf. \cite[Remark 3.3.2]{PS}). Thus it is enough to check the rank index of any smooth $k$-dimensional linear section for each $3 \leq k \leq 5$.

\begin{theorem}\label{thm:smooth del Pezzo degree 5}
Let $X_k \subset \P^{k+3} \quad (3 \leq k \leq 5)$ be as above. Then
$$\emph{rank-index}(X_k)  = \begin{cases} 6 & \mbox{if} \quad $k=4,5$, \quad \mbox{and}\\
                                          5 & \mbox{if} \quad $k=3$. \end{cases}$$
\end{theorem}

\begin{proof}
Following the notation in Theorem \ref{thm:Grassmannian}, let $\{p_{ij} ~|~ 0 \le i < j \le 4 \}$ be the set of homogeneous coordinates of the projective space $\P^9$. Then the homogeneous ideal of $\mathbb{G}(1,\mathbb{P}^4)$ is generated by the following five Pl\"{u}cker relations
\begin{equation*}
Q(i,j,k,l) = p_{ij} p_{kl} - p_{ik} p_{jl} + p_{jk} p_{il}  \quad (0 \le i < j < k < l \le 4).
\end{equation*}
Now, let $H_1$, $H_2$ and $H_3$ be hyperplanes of $\P^9$ which are defined respectively by
\begin{equation*}
p_{34}-p_{01}-p_{02} , ~ p_{24}-p_{03}-p_{04} \quad \mbox{and} \quad p_{23}-p_{12}-p_{13}-p_{14} .
\end{equation*}
Then one can check that $X_5 := \mathbb{G} (1,\mathbb{P}^4) \cap H_1$, $X_4 := X_5 \cap H_2$ and $X_3 := X_4 \cap H_3$ are all irreducible and smooth. Furthermore, we can find the rank index of $X_k$ using the restrictions of the above five Pl\"{u}cker relations to the projective spaces $\langle X_k \rangle$ (cf. Proposition \ref{prop:computation of the rank index}). For example, the homogeneous ideal of $X_5$ is generated by $\{Q^{(1)}(i,j,k,l)\}_{0\leq i<j<k<l\leq 4}$ which is the substitution of the Pl\"{u}cker relations by $p_{34}=p_{01}+p_{02}$.
Now, let
\begin{equation*}
Q^{(1)} = \sum_{0 \le i < j < k < l \le 4} x^{(1)}_{i,j,k,l} Q^{(1)}(i,j,k,l)
\end{equation*}
be a quadric element in the homogeneous ideal of $X_5$. Using Macaulay2, we can show that
$$\sqrt{I(6,M_{Q^{(1)}})}=(x_{0123}^{(1)},x_{0124}^{(1)},x_{0234}^{(1)}-x_{0134}^{(1)},x_{1234}^{(1)}).$$
Since it contains a linear form, we conclude that the rank-index of $X_5$ is strictly bigger than $5$ and hence is equal to $6$. By the same way, we can show that the rank indices of $X_4$ and $X_3$ are respectively $6$ and $5$.
\end{proof}

\begin{remark}
Theorem \ref{thm:smooth del Pezzo degree 5} is an interesting example of how the rank index can be changed by the hyperplane section. Let $X_6 := \mathbb{G} (1,\P^4 )$ and let $X_2$ be a general hyperplane section of $X_3$. Thus $X_2$ is projectively equivalent to $S_4$ which is defined in the previous subsection. Also let $X_1$ be a general hyperplane section of $X_2$. Thus $X_1 \subset \P^4$ is an elliptic normal curve. Then it follows by \cite[Theorem 1.1]{P}, Theorem \ref{thm:smooth del Pezzo surface} and Theorem \ref{thm:smooth del Pezzo degree 5} that
$$\mbox{rank-index}(X_k)  = \begin{cases} 6 & \mbox{if} \quad $k=4,5,6$,  \\
                                          5 & \mbox{if} \quad $k=3$, \\
                                          4 & \mbox{if} \quad $k=2$, \quad \mbox{and} \\
                                          3 & \mbox{if} \quad $k=1$. \\ \end{cases}$$
\end{remark}
\end{description}

\subsection{Non-normal del Pezzo varieties of dimension $\geq 2$}
\noindent Let $X \subset \P^{n+e}$ be an $n$-dimensional non-normal del Pezzo variety which is not a cone, where $n \geq 2$ and $e \geq 3$. First we quickly review the known classification results of $X$.

\begin{notation and remark}
Let $X \subset \P^{n+e}$ be as above. Then
\smallskip

\renewcommand{\descriptionlabel}[1]%
             {\hspace{\labelsep}\textrm{#1}}
\begin{description}
\setlength{\labelwidth}{13mm}
\setlength{\labelsep}{1.5mm}
\setlength{\itemindent}{0mm}

\item[{\rm (1)}] By \cite[Theorem 6.2]{BP}, there is a smooth rational normal scroll $\tilde{X} \subset \P^{n+e+1}$ of dimension $\leq 3$ such that $X$ is the linear projection of $\tilde{X}$ from a point. More precisely, if $n=2$ then $X = \pi_p (\tilde{X})$ where
\begin{equation*}
\quad \quad \quad \tilde{X} = S(1,e+1) \mbox{ and $p \in \langle S(1) , \mathbb{L} \rangle - S(1) \cup \mathbb{L}$ for a ruling $\mathbb{L}$ of $S(1,e+1)$}
\end{equation*}
or else
\begin{equation*}
\tilde{X} = S(2,e) \quad \mbox{and $p \in \langle S(2)   \rangle - S(2)$}.
\end{equation*}
Also if $n=3$ then $X = \pi_q (\tilde{X})$ where
\begin{equation*}
\tilde{X} = S(1,1,e) \quad \mbox{and $q \in \langle S(1,1)  \rangle - S(1,1)$.}
\end{equation*}
Note that a general hyperplane section of $\pi_q ( S(1,1,e)) \subset \P^{e+3}$ is projectively equivalent to $\pi_p (S(2,e)) \subset \P^{e+2}$.

\item[{\rm (2)}] Let $R = \mathbb{K} [z_0 , \ldots ,z_{e+2} ]$ be the homogeneous coordinate ring of $\P^{e+2}$ and let
$$A_1 = \{ Q_{i,j} := z_i z_{j+1} - z_{i+1} z_j ~|~ 2 \leq i < j \leq e+1 \}$$
be the set of all $2 \times 2$ minors of the matrix
$\begin{bmatrix}
z_{2} & z_3   & \cdots & z_{e+1} \\
z_3   & z_{4} & \cdots & z_{e+2}
\end{bmatrix}$.
Also let
\begin{equation*}
B_1 = \{ F_i := z_2 z_{i-1} + z_1 z_{i} - z_0 z_{i+1} ~|~ 3 \leq i \leq e+1 \}
\end{equation*}
and
\begin{equation*}
B_2 = \{ G_i := z_1 z_{i-1} - z_0 z_{i+1} ~|~ 3 \leq i \leq e+1 \}.
\end{equation*}
Then $A_1 \cup B_1$ and $A_1 \cup B_2$ are $\Bbbk$-linearly independent. Also the homogeneous ideals $\wp_1 := \langle A_1 , B_1 \rangle$ and $\wp_2 := \langle A_1 , B_2 \rangle$ of $R$ are prime, and
\begin{equation*}
{S}_1 := V(\wp_1 ) \subset \P^{e+2} \quad \mbox{and} \quad {S}_2 := V(\wp_2 ) \subset \P^{e+2}
\end{equation*}
are non-normal del Pezzo surfaces. Furthermore, ${S}_1$ and ${S}_2$ are respectively projectively equivalent to $\pi_p (S(1,d-1))$ and $\pi_p (S(2,d-2))$. For details, we refer the reader to \cite[Proposition 3.5 and 3.6]{LP} and \cite[Theorem 4.6 and 4.7]{LP}.

\item[{\rm (3)}] Let $R = \mathbb{K} [z_0 , \ldots ,z_{e+2} , z_{e+3} ]$ be the homogeneous coordinate ring of $\P^{e+3}$ and let
$$A_2 = \{ Q_{i,j} := z_i z_{j+1} - z_{i+1} z_j ~|~ 3 \leq i < j \leq e+2 \}$$
be the set of all $2 \times 2$ minors of the matrix
$\begin{bmatrix}
z_{3} & z_4   & \cdots & z_{e+2} \\
z_4   & z_{5} & \cdots & z_{e+3}
\end{bmatrix}$. Also let
\begin{equation*}
C = \{ H_i := z_1 z_{i-1} + z_0 z_{i} - z_2 z_{i+1} ~|~ 4 \leq i \leq e+2 \}.
\end{equation*}
Then $A_2 \cup C$ is $\Bbbk$-linearly independent. Also the homogeneous ideals $\wp_3 := \langle A_2, C \rangle$ of $R$ is prime, and
\begin{equation*}
T := V(\wp_3 ) \subset \P^{e+3}
\end{equation*}
is a non-normal del Pezzo threefold. Furthermore, $T$ is projectively equivalent to $\pi_p (S(1,1,e+1))$. For details, we refer the reader to \cite[Proposition 3.8 and Theorem 4.8]{LP}.
\end{description}
\end{notation and remark}

\begin{theorem}\label{thm:non-normal del Pezzo}
Let ${S}_1 ,{S}_2 \subset \P^{e+2}$ and $T  \subset \P^{e+3}$ be as above. Then
\smallskip

\renewcommand{\descriptionlabel}[1]%
             {\hspace{\labelsep}\textrm{#1}}
\begin{description}
\setlength{\labelwidth}{13mm}
\setlength{\labelsep}{1.5mm}
\setlength{\itemindent}{0mm}

\item[{\rm (1)}] $\emph{rank-index}(S_1) = 5$.
\item[{\rm (2)}] $\emph{rank-index}({S}_2 ) = 4$.
\item[{\rm (3)}] $\emph{rank-index}(T) = 6$.
\end{description}
\end{theorem}

\begin{proof}
$(1)$ In order to prove that $\mbox{rank-index}( {S}_1 ) \geq 5$, it needs to find a nonzero linear form in the homogeneous ideal of $\Phi_4 ({S}_1 )$ (cf. Proposition \ref{prop:computation of the rank index}). Let
\begin{equation*}
Q = \sum_{0 \leq i < j \leq e+2} \alpha_{i,j} z_i z_j
\end{equation*}
be a nonzero element of $I({S}_1)$. It can be written as
\begin{equation*}
Q = \sum_{2 \leq i < j \leq e+1} a_{i,j} Q_{i,j} + \sum_{3 \leq i \leq e+1} b_i F_i
\end{equation*}
for some $a_{i,j}$'s and $b_i$'s in $\mathbb{K}$. Using this second expression, one can easily check that
\begin{equation*}
\alpha_{i,j} = 0 \quad \mbox{whenever} \quad i+j \leq 3.
\end{equation*}
Let $M$ be the $(e+3) \times (e+3)$ symmetric matrix associated to $Q$. Thus $M_{i,j} = \alpha_{i,j}$ if $i <j$ and $M_{i,i} = 2 \alpha_{i,i}$. Now, consider the following $5 \times 5$ sub-matrix $N$ of $M$:
$$N = \begin{bmatrix}
2 \alpha_{0,0} & \alpha_{0,1}   & \alpha_{0,2}   & \alpha_{0,3}   & \alpha_{0,4} \\
\alpha_{0,1}   & 2 \alpha_{1,1} & \alpha_{1,2}   & \alpha_{1,3}   & \alpha_{1,4} \\
\alpha_{0,2}   & \alpha_{1,2}   & 2 \alpha_{2,2} & \alpha_{2,3}   & \alpha_{2,4}\\
\alpha_{0,3}   & \alpha_{1,3}   & \alpha_{2,3}   & 2 \alpha_{3,3} & \alpha_{3,4}\\
\alpha_{0,4}   & \alpha_{1,4}   & \alpha_{2,4}   & \alpha_{3,4}   & 2 \alpha_{4,4}
\end{bmatrix} = \begin{bmatrix}
0   & 0            & 0    & 0   & -b_3 \\
0   & 0            & 0    & b_3 & \alpha_{1,4} \\
0   & 0            & 2b_3 & \alpha_{2,3}   & \alpha_{2,4}\\
0   & b_3          & \alpha_{2,3}   & 2 \alpha_{3,3} & \alpha_{3,4}\\
-b_3 & \alpha_{1,4} & \alpha_{2,4}   & \alpha_{3,4}   & 2 \alpha_{4,4}
\end{bmatrix}$$
Then
\begin{equation*}
\det (N) = 2b_3 ^5 \in I(5,M) \quad \mbox{and hence} \quad b_3 \in \sqrt{I(5,M)} = I(\Phi_4 (  {S}_1 )).
\end{equation*}
Therefore the rank index of $  {S}_1$ is at least $5$. To prove that $  {S}_1$ is at most $5$, we consider a general hyperplane section $\mathcal{C}$ of $S_1$. Note that $\mathcal{C} \subset \P^{d-1}$ is a linearly normal curve of arithmetic genus $1$. Then it holds by \cite[Theorem 1.1]{P} that $\mbox{rank-index}(\mathcal{C}) = 3$. Then we get the desired inequality
$$\mbox{rank-index}(S_1 ) \leq \mbox{rank-index}(\mathcal{C}) +2 =5$$
by Proposition \ref{prop:hyperplane section}. This completes the proof that $\mbox{rank-index}(   {S}_1 ) = 5$. \smallskip
\smallskip

\noindent $(2)$ Since $I(  {S}_2 ) = \wp_2$ is generated by $A_1 \cup B_2$ and the rank of elements in $A_1 \cup B_2$ is at most $4$, it follows immediately that $\mbox{rank-index}(   {S}_2 ) \leq 4$. To show that $\mbox{rank-index}(   {S}_2 ) \geq 4$, it needs to find a nonzero linear form in the homogeneous ideal of $\Phi_3 (  {S}_2 )$ (cf. Proposition \ref{prop:computation of the rank index}). Let
\begin{equation*}
Q = \sum_{0 \leq i < j \leq e+2} \alpha_{i,j} z_i z_j
\end{equation*}
be a nonzero element of $I(  {S}_2)$. It can be written as
\begin{equation*}
Q = \sum_{2 \leq i < j \leq e+1} a_{i,j} Q_{i,j} + \sum_{3 \leq i \leq e+1} b_i G_i
\end{equation*}
for some $a_{i,j}$'s and $b_i$'s in $\mathbb{K}$. Using this second expression, one can easily check that
\begin{equation*}
\alpha_{i,j} = 0 \quad \mbox{whenever} \quad i+j \leq 2.
\end{equation*}
Again, let $M$ be the $(e+3) \times (e+3)$ symmetric matrix associated to $Q$. From the following $4 \times 4$ sub-matrix
$$N = \begin{bmatrix}
2 \alpha_{0,0} & \alpha_{0,1}   & \alpha_{0,2}   & \alpha_{0,4}     \\
\alpha_{0,1}   & 2 \alpha_{1,1} & \alpha_{1,2}   & \alpha_{1,4} \\
\alpha_{0,2}   & \alpha_{1,2}   & 2 \alpha_{2,2} & \alpha_{2,4}\\
\alpha_{0,4}   & \alpha_{1,4}   & \alpha_{2,4}   & 2 \alpha_{4,4}
\end{bmatrix} = \begin{bmatrix}
0    & 0   & 0            & -b_3 \\
0    & 0   & b_3          & b_5 \\
0    & b_3 & 0            & \alpha_{2,4}\\
-b_3 & b_5 & \alpha_{2,4} & 2 \alpha_{4,4}
\end{bmatrix}$$
of $M$, it holds that
\begin{equation*}
\det (N) = b_3 ^4 \in I(4,M) \quad  \mbox{and hence} \quad b_3 \in \sqrt{I(4,M)} = I(\Phi_3 (  {S}_2 )).
\end{equation*}
This shows that the rank index of $  {S}_2$ is at least $4$ (cf. Proposition \ref{prop:computation of the rank index}).
\smallskip

\noindent $(3)$ Since $I(T) = \wp_3$ is generated by $A_2 \cup C$ and the rank of elements in $A_2 \cup C$ is at most $6$, it follows immediately that
the rank index of $T$ is at most $6$. To show the inequality $\mbox{rank-index}(T) \leq 6$, it needs to find a nonzero linear form in the homogeneous ideal of $\Phi_5 (T)$ (cf. Proposition \ref{prop:computation of the rank index}). Let
\begin{equation*}
Q = \sum_{0 \leq i < j \leq e+3} \alpha_{i,j} z_i z_j
\end{equation*}
be a nonzero element of $I(T)$. It can be written as
\begin{equation*}
Q = \sum_{3 \leq i < j \leq e+2} a_{i,j} Q_{i,j} + \sum_{4 \leq i \leq e+2} b_i H_i
\end{equation*}
for some $a_{i,j}$'s and $b_i$'s in $\mathbb{K}$. Again, let $M$ be the $(e+4) \times (e+4)$ symmetric matrix associated to $Q$. From the following $6 \times 6$ sub-matrix
$$N = \begin{bmatrix}
2 \alpha_{0,0} & \alpha_{0,1}   & \alpha_{0,2}   & \alpha_{0,3}  &  \alpha_{0,4}   & \alpha_{0,5}   \\
\alpha_{0,1}   & 2 \alpha_{1,1} & \alpha_{1,2}   & \alpha_{1,3}  &  \alpha_{1,4}   & \alpha_{1,5} \\
\alpha_{0,2}   & \alpha_{1,2}   & 2 \alpha_{2,2} & \alpha_{2,3}  &  \alpha_{2,4}   & \alpha_{2,5} \\
\alpha_{0,3}   & \alpha_{1,3}   &   \alpha_{2,3} & 2 \alpha_{3,3}  &  \alpha_{3,4}   & \alpha_{3,5} \\
\alpha_{0,4}   & \alpha_{1,4}   &   \alpha_{2,4} & \alpha_{3,4}  & 2 \alpha_{4,4}   & \alpha_{4,5} \\
\alpha_{0,5}   & \alpha_{1,5}   &   \alpha_{2,5} & \alpha_{3,5}  &  \alpha_{4,5}   & 2 \alpha_{5,5} \\
\end{bmatrix} = \begin{bmatrix}
0    & 0   & 0    & 0       & b_4       & b_5 \\
0    & 0   & 0    & b_4     & b_5       & b_6  \\
0    & 0   & 0    &0        & 0         & -b_4 \\
0    & b_4 & 0    &0        & 0         & a_{3,4} \\
b_4  & b_5 & 0    &0        & -2a_{3,4} & a_{3,5} \\
b_5  & b_6 & -b_4 &a_{3,4}  & a_{3,5}   & 2 a_{4,5}
\end{bmatrix}$$
of $M$, it holds that
\begin{equation*}
\det (N) = -b_4 ^6 \in I(5,M) \quad  \mbox{and hence} \quad b_4 \in \sqrt{I(5,M)} = I(\Phi_5 (T)).
\end{equation*}
By Proposition \ref{prop:computation of the rank index}, this completes the proof that the rank index of $T$ is equal to $6$.
\end{proof}

\end{document}